\documentclass[12pt,reqno]{amsart}
\usepackage{amsthm}
\usepackage{amssymb}
\usepackage{graphics}
\usepackage{tikz}
\usetikzlibrary{shapes,backgrounds,calc}
\usepackage{latexsym}
\usepackage{multicol}
\usepackage{verbatim,enumerate}
\usepackage{accents}
\usepackage{cite}
\usepackage{multirow}
\usepackage{bigstrut}
\usepackage{amsthm}
\usepackage{amssymb}
\usepackage{graphics}
\usepackage{tikz}
\usetikzlibrary{shapes,backgrounds,calc}
\usepackage{latexsym}
\usepackage{multicol}
\usepackage{verbatim,enumerate}
\usepackage{accents}
\usepackage{cite}
\usepackage{array}
\usepackage[colorlinks=true, linkcolor=blue, citecolor=blue, urlcolor=black]{hyperref}
\usepackage{hyperref}
\usepackage{amsmath, amscd,url}
\usepackage{multirow}
\usepackage{bigstrut}
\usepackage{longtable}

\usepackage{stackengine}

\usepackage{hyperref}
\usepackage{amsmath, amscd,url}
\usepackage{longtable}

\advance\textwidth by 1.3in \advance\oddsidemargin by -.6in \advance\evensidemargin by -.6in
\theoremstyle{definition}

\newtheorem{theorem}{Theorem}[section]
\newtheorem{lemma}[theorem]{Lemma}

\theoremstyle{definition}
\newtheorem{definition}[theorem]{Definition}
\newtheorem{example}[theorem]{Example}

\theoremstyle{remark}
\newtheorem{remark}[theorem]{Remark}

\theoremstyle{definition}

\newcounter{cnt}
 \makeatletter
\def\mydggeometry{\makeatletter\dg@YGRID=1\dg@XGRID=20\unitlength=0.003pt\makeatother}
\makeatother \theoremstyle{remark}

\numberwithin{equation}{section}
\let\bwdg\bigwedge
\def\bigwedge{{\textstyle\bwdg}}

\newcommand{\nc}{\newcommand}
\newcommand{\rnc}{\renewcommand}

\nc{\cal}{\mathcal} \nc{\goth}{\mathfrak} \rnc{\bold}{\mathbf}

\nc\bomega{{\mbox{\boldmath $\omega$}}} \nc\bpsi{{\mbox{\boldmath $\Psi$}}}
 \nc\balpha{{\mbox{\boldmath $\alpha$}}}
 \nc\bpi{{\mbox{\boldmath $\pi$}}}
 \nc\bvpi{{\mbox{\boldmath $\varpi$}}}
\nc\chara{\operatorname{ch}}

  \nc\bxi{{\mbox{\boldmath $\xi$}}}
\nc\bmu{{\mbox{\boldmath $\mu$}}} \nc\bcN{{\mbox{\boldmath $\cal{N}$}}} \nc\bcm{{\mbox{\boldmath $\cal{M}$}}} \nc\blambda{{\mbox{\boldmath
$\lambda$}}}\nc\bnu{{\mbox{\boldmath $\nu$}}}

\makeatletter
\def\section{\def\@secnumfont{\mdseries}\@startsection{section}{1}%
  \z@{.7\linespacing\@plus\linespacing}{.5\linespacing}%
  {\normalfont\scshape\centering}}
\def\subsection{\def\@secnumfont{\bfseries}\@startsection{subsection}{2}%
  {\parindent}{.5\linespacing\@plus.7\linespacing}{-.5em}%
  {\normalfont\bfseries}}
\makeatother

 \nc{\Hom}{\operatorname{Hom}}
  \nc{\mode}{\operatorname{mod}}
\nc{\End}{\operatorname{End}} \nc{\wh}[1]{\widehat{#1}} \nc{\Ext}{\operatorname{Ext}} \nc{\ch}{\text{ch}} \nc{\ev}{\operatorname{ev}}
\nc{\Ob}{\operatorname{Ob}} \nc{\soc}{\operatorname{soc}} \nc{\rad}{\operatorname{rad}} \nc{\head}{\operatorname{head}}

 \nc{\Cal}{\cal} \nc{\Xp}[1]{X^+(#1)} \nc{\Xm}[1]{X^-(#1)}
\nc{\on}{\operatorname} \nc{\Z}{{\bold Z}} \nc{\J}{{\cal J}}  \nc{\Q}{{\bold Q}}

\nc{\N}{{\bold N}}  \nc\boa{\bold a} \nc\bob{\bold b} \nc\boc{\bold c} \nc\bod{\bold d} \nc\boe{\bold e} \nc\bof{\bold f} \nc\bog{\bold g}
\nc\boh{\bold h} \nc\boi{\bold i} \nc\boj{\bold j} \nc\bok{\bold k} \nc\bol{\bold l} \nc\bom{\bold m} \nc\bon{\mathbb n} \nc\boo{\bold o}
\nc\bop{\bold p} \nc\boq{\bold q} \nc\bor{\bold r} \nc\bos{\bold s} \nc\boT{\bold t} \nc\boF{\bold F} \nc\bou{\bold u} \nc\bov{\bold v}
\nc\bow{\bold w} \nc\boz{\bold z}\nc\ba{\bold A} \nc\bb{\bold B} \nc\bc{\mathbb C} \nc\bd{\bold D} \nc\be{\bold E} \nc\bg{\bold
G} \nc\bh{\bold H} \nc\bi{\bold I} \nc\bj{\bold J} \nc\bk{\bold K} \nc\bl{\bold L} \nc\bm{\bold M} \nc\bn{\mathbb N} \nc\bo{\bold O} \nc\bp{\bold
P} \nc\bq{\bold Q} \nc\br{\bold R} \nc\bs{\bold S} \nc\bt{\bold T} \nc\bu{\bold U} \nc\bv{\bold V} \nc\bw{\bold W} \nc\bz{\mathbb Z} \nc\bx{\bold
x} \nc\KR{\bold{KR}} \nc\rk{\bold{rk}} \nc\het{\text{ht }}

\nc\toa{\tilde a} \nc\tob{\tilde b} \nc\toc{\tilde c} \nc\tod{\tilde d} \nc\toe{\tilde e} \nc\tof{\tilde f} \nc\tog{\tilde g} \nc\toh{\tilde h}
\nc\toi{\tilde i} \nc\toj{\tilde j} \nc\tok{\tilde k} \nc\tol{\tilde l} \nc\tom{\tilde m} \nc\ton{\tilde n} \nc\too{\tilde o} \nc\toq{\tilde q}
\nc\tor{\tilde r} \nc\tos{\tilde s} \nc\toT{\tilde t} \nc\tou{\tilde u} \nc\tov{\tilde v} \nc\tow{\tilde w} \nc\toz{\tilde z} \nc\woi{w_{\omega_i}}

\begin{document}
\setcounter{section}{0}
\setcounter{tocdepth}{1}


\title{An extension of a second irreducibility theorem of I. Schur}

\author[Anuj Jakhar]{Anuj Jakhar}
\author[Ravi Kalwaniya]{Ravi Kalwaniya}
\address[]{Department of Mathematics, Indian Institute of Technology (IIT) Madras}

\email[Anuj Jakhar]{anujjakhar@iitm.ac.in \\ anujiisermohali@gmail.com}
\email[Ravi Kalwaniya]{ravikalwaniya3@gmail.com}



\begin{abstract}
\noindent Let $n \neq 8$ be a positive integer such that $n+1 \neq 2^u$ for any integer $u\geq 2$. Let $\phi(x)$ belonging to $\Z[x]$ be a monic polynomial which is irreducible modulo all primes less than or equal to  $n+1$. Let $a_j(x)$ with $0\leq j\leq n-1$ belonging to $\Z[x]$  be polynomials having degree less than $\deg\phi(x)$. Assume that the content of $(a_na_0(x))$ is not divisible by any prime less than or equal to $n+1$. In this paper, we prove that the polynomial $f(x) = a_n\frac{\phi(x)^n}{(n+1)!}+ \sum\limits_{j=0}^{n-1}a_j(x)\frac{\phi(x)^{j}}{(j+1)!}$ is irreducible over the field $\Q$ of rational numbers. This generalises a well-known result of Schur which states that the polynomial   $\sum\limits_{j=0}^{n}a_j\frac{x^{j}}{(j+1)!}$ with $a_j \in \Z$ and $|a_0| = |a_n| = 1$ is irreducible over $\Q$. We illustrate our result through examples.
\end{abstract}
\maketitle

\section{Introduction and statements of results}\label{intro}

In 1929, Schur\cite{Sch 1} proved the following result.
\begin{theorem}
	Let $n \neq 8$ be a positive integer such that $n+1 \neq 2^u$ for any integer $u\geq 2$.  Let $a_0, a_1, \cdots, a_n$ be integers with $|a_0| = |a_n| = 1$. Then the polynomial   $\sum\limits_{j=0}^{n}a_j\frac{x^{j}}{(j+1)!}$  is irreducible over the field $\Q$ of rationals.
\end{theorem}
In the present paper, we extend this result in a different direction using $\phi$-Newton polygons (defined in the next section as Definition \ref{def1}). Precisely, we prove the following.


\begin{theorem}\label{1.1}
	Let $n \neq 8$ be a positive integer such that $n+1 \neq 2^u$ for any integer $u\geq 2$.  Let $\phi(x)$ belonging to $\Z[x]$ be a monic polynomial which is irreducible modulo all primes less than or equal to $n+1$. Suppose that $a_n, a_0(x), \cdots, a_{n-1}(x)$ belonging to $\Z[x]$ satisfy the following conditions.
	\begin{itemize}
\item[(i)] $\deg a_j(x) < \deg \phi(x)$ for $0\leq j\leq n-1$,
\item[(ii)] the content of $(a_na_0(x))$ is not divisible by any prime less than or equal to $n+1$.
	\end{itemize}
	Then the polynomial $$f(x) = a_n\frac{\phi(x)^n}{(n+1)!}+ \sum\limits_{j=0}^{n-1}a_j(x)\frac{\phi(x)^{j}}{(j+1)!}$$  is irreducible over $\Q$. 
\end{theorem}

It may be pointed out that in the above theorem the assumptions ``$n+1 \neq 2^u$ for any integer $u\geq 2$" and ``the content of $a_0(x)$ is not divisible by any prime less than or equal to $n+1$" cannot be dispensed with. For example, consider the polynomial $\phi(x)=x^3-x+7$ which is irreducible modulo $2$, $3$ and $5$. Then for $n = 3$, the polynomial $f(x) = \frac{\phi(x)^3}{4!}+\frac{\phi(x)^2}{3!}-1$ = $\frac{1}{4!}(\phi(x)-2)(\phi(x)^2+6\phi(x)+12)$ is reducible over $\Q$.  Similarly, the polynomial $f(x)=\frac{\phi(x)^4}{5!}+12\frac{\phi(x)^2}{3!}+120=\frac{1}{5!}(\phi(x)^2+120)^2$ is reducible over $\Q.$ 

We also give below an example to show that Theorem $\ref{1.1}$ may not hold if $a_n$ is replaced by a (monic) polynomial $a_n(x)$ with integer coefficient having degree less than $\deg\phi(x)$. Consider $\phi(x) = x^3-x+7$ which is irreducible modulo $2$, $3$ and $5$. Take $a_4(x) = x+62, a_3(x) = x-2, a_2(x) = x+3 = a_1(x)$ and $a_0(x) = x+1$. Then the polynomial $$a_4(x)\frac{\phi(x)^4}{5!} + a_3(x)\frac{\phi(x)^3}{4!} + a_2(x)\frac{\phi(x)^2}{3!} + a_1(x)\frac{\phi(x)}{2!} + a_0(x)$$ has $-2$ as a root.

\vspace{0.2cm}

We also prove the following theorem which is of independent interest as well. 
It will be used in the proof of Theorem \ref{1.1}.
\begin{theorem}\label{1.2}
	Let $\phi(x)\in \Z[x]$ be a monic polynomial which is irreducible modulo all primes less than or equal to $n+1$ and $a_0(x), a_1(x), \cdots, a_n(x)$ belonging to $\Z[x]$ be polynomials each having degree less than $\deg \phi(x)$. Let the content of $a_0(x)$ is not divisible by any prime less than or equal to $n+1$. Let $$f(x) = \sum\limits_{j=0}^{n}a_j(x)\frac{\phi(x)^{j}}{(j+1)!}.$$ Let $k$ be a positive integer $\leq \frac{n}{2}$. Suppose that there exists a prime $p \geq k+2$ such that $$p| (n+1)n(n-1)\cdots (n-k+2)$$ and $p$ does not divide the leading coefficient of $a_n(x)$. Then $f(x)$ can not have a factor in $\Z[x]$ with degree lying in the interval $[k\deg\phi(x), (k+1)\deg\phi(x)).$
\end{theorem}
As an application, we provide a class of examples of Theorem \ref{1.1}. 

\begin{example}
	Consider $\phi(x) = x^4-x-1$. It can be easily checked that $\phi(x)$ is irreducible modulo $2,3$ and $5$. Let $j\geq 2$ and $a_j$ be integers. Let $a_{i}(x) \in \Z[x]$ be polynomials each having degree less than $4$ for $0\leq i\leq j-1$. Assume that the content of $(a_ja_0(x))$ is not divisible by any prime less than or equal to $j+1$. Then by Theorem \ref{1.1}, the polynomial	$$f_j(x) = \frac{a_{j}}{(j+1)!}\phi(x)^{j}+\sum\limits_{i=0}^{j-1}a_i(x)\frac{\phi(x)^{i}}{(i+1)!}$$  is irreducible over $\Q$ for $4\leq j\leq 6$.
\end{example}
It may be pointed out that, in the above example, the irreducibility of the polynomial $f_5(x)$ does not seem to follow from any known irreducibility criterion (cf. \cite{brown2008}, \cite{JakBLMS}, \cite{JakAMS}, \cite{JakJA}, \cite{JakAM}, \cite{Ja-Sa} and \cite{Jho-Kh}).
\section{Proof of Theorem \ref{1.2}}
Before proving Theorem \ref{1.2}, we shall first introduce the notion of Gauss valuation and $\phi$-Newton polygon. For a prime $p$, $v_p$ will denote the $p$-adic valuation of $\Q$ defined for any non-zero integer $b$ to be the highest power of $p$ dividing $b$. 
We shall denote by $v_p^x$ the Gaussian valuation extending $v_p$ defined on the polynomial ring $\Z[x]$ by $$v_p^x(\sum\limits_{i}b_ix^i) = \min_{i}\{v_p(b_i)\}, b_i \in \Z.$$ We now define the notion of $\phi$-Newton polygon.
\begin{definition}\label{def1}
	Let $p$ be a prime number and $\phi(x) \in \Z[x]$ be a monic polynomial which is irreducible modulo $p$. Let $f(x)$ belonging to $\Z[x]$ be a polynomial having $\phi$-expansion\footnote{If $\phi(x)$ is a fixed monic polynomial with coefficients in $\Z$, then any $f(x)\in \Z[x]$ can be uniquely written as a finite sum $\sum\limits_{i}b_i(x)\phi(x)^i$ with $\deg b_i(x)< \deg \phi(x)$ for each $i$; this expansion will be referred to as the $\phi$-expansion of $f(x)$.} $\sum\limits_{i=0}^{n}b_i(x) \phi(x)^i$ with $b_0(x)b_n(x) \neq 0$. Let $P_i$ stand for the point in the plane having coordinates $(i, v_p^x(b_{n-i}(x)))$ when $b_{n-i}(x)\neq 0$, $0\leq i \leq n$.
Let $\mu_{ij}$ denote the slope of the line joining the points $P_i$ and $P_j$ if $b_{n-i}(x)b_{n-j}(x)\ne 0$. Let $i_1$ be the largest index $0< i_1 \leq n$ such that 
\begin{align*}
	\mu_{0i_1}=\min \{\mu_{0j}\ |\ 0<j \leq n,\ b_{n-j}(x)\ne 0\}.
\end{align*}
If $i_1<n$, let $i_2$ be the largest index $i_1< i_2 \leq n$ such that 
\begin{align*}
	\mu_{i_1i_2}=\min \{\mu_{i_1j}\ |\ i_1<j \leq n,\ b_{n-j}(x)\ne 0\}
\end{align*}
and so on. The $\phi$-Newton polygon of $f(x)$ with respect to $p$ is the polygonal path having segments $P_0P_{i_1}, P_{i_1}P_{i_2}, \dots, P_{i_{k-1}}P_{i_k}$ with $i_k=n$. These segments are called the edges of the $\phi$-Newton polygon of $f(x)$ and their slopes from left to right form a strictly increasing sequence. The $\phi$-Newton polygon minus the horizontal part (if any) is called its principal part.
\end{definition}

We shall use the following lemma\cite[Lemma 2.4]{Ji-Kh} in the proof of Theorem \ref{1.2}.
\begin{lemma}\label{lm1} Let $\phi(x)$ belonging to $\Z[x]$ be a monic polynomial which is irreducible modulo a given prime $p$. Let $g(x), h(x)$ belonging to $\Z[x]$ be polynomials not divisible by $\phi(x)$ having leading coefficients coprime with $p$. Then the following hold.
\begin{itemize}
	\item [(i)] The principal part of the $\phi$-Newton polygon of $g(x)h(x)$ with respect to $p$ is obtained by constructing a polygonal path beginning with a point on the non-negative side of $x$-axis and using translates of edges of the principal part in the $\phi$-Newton polygons of $g(x), h(x)$ in the increasing order of slopes.
	\item [(ii)] The length of the edge of the $\phi$-Newton polygon of $g(x)h(x)$ with respect to $p$ having slope zero is either the sum of the lengths of the edges of the $\phi$-Newton polygons of $g(x)$, $h(x)$ with respect to $p$ having slope zero or it exceeds this sum by one.
\end{itemize}
\end{lemma}
\begin{proof}[Proof of Theorem \ref{1.2}]
	To prove that $f(x)$ can not have a factor in $\Z[x]$ with degree lying in the interval $[k\deg\phi(x), (k+1)\deg\phi(x))$, it suffices to show that $F(x) = (n+1)!f(x)$ can not have a factor in $\Z[x]$ with degree lying in the interval $[k\deg\phi(x), (k+1)\deg\phi(x))$. Let $b_j$ denote the integer such that $b_j = \frac{(n+1)!}{(j+1)!}$ for $0\leq j \leq n$.
	Then $$F(x) = (n+1)!f(x) = \sum\limits_{j=0}^{n}b_ja_j(x)\phi(x)^{j}.$$
	Note that the condition $p|(n+1)n(n-1)\cdots (n-k+2)$ implies that $p|b_j$ for $j \in \{0, 1, \cdots, n-k\}$. Thus, the $n-k+1$ rightmost points,
	$$(k, v_{p}^{x}(b_{n-k}a_{n-k}(x)), \cdots, (n-1, v_p^x(b_1a_{1}(x)), (n, v_p^x(b_0a_{0}(x)),$$
	associated with the $\phi$-Newton polygon of $F(x)$ with respect to $p$ have $y$-ordinates $\geq 1$. Note that the leftmost endpoint is given by $(0, v_p^x(a_n(x))$. Since the leading coefficient of $a_n(x)$ is not divisible by $p$, we have $v_p^x(a_n(x))=0$. Hence the leftmost endpoint of the $\phi$-Newton polygon of $F(x)$ with respect to $p$ will  be $(0,0)$. 
	By definition of $\phi$-Newton polygon of $F(x)$, we know that the slopes of the edges increase from left to right, hence the points $(j, v_p^x(b_{n-j}a_{n-j}(x))$ for $j\in \{k-1, k, k+1, \cdots, n\}$ all lie on or above edges of the Newton polygon which have positive slope. We claim that each of these positive slopes is strictly less than $\frac{1}{k}.$ Observe that the slope of the rightmost edge is given by
	$$\max\limits_{1\leq j\leq n}\bigg\{\frac{v_p^x(b_0a_0(x))-v_p^x(b_ja_j(x))}{j}\bigg\}.$$
	For $1\leq j\leq n$, we have
	\begin{align*}
		v_p^x(b_0a_0(x))-v_p^x(b_ja_j(x)) &= v_p^x((n+1)!a_0(x))-v_p^x\bigg(\frac{(n+1)!}{(j+1)!}a_j(x)\bigg)\\ & \leq v_p((n+1)!) - v_p\bigg(\frac{(n+1)!}{(j+1)!}\bigg) = v_p((j+1)!).
	\end{align*}
	We consider two cases to estimate $v_p((j+1)!)/j.$
	
	Case (i): Suppose that $j < p-1$. Since $p$ is prime number and $j+1 < p$, we have $p\nmid (j+1)!$. Therefore, $v_p((j+1)!) = 0$. Hence for $j<p-1$,
	$$\frac{v_p((j+1)!)}{j} = 0.$$
	
	Case(ii): Suppose that $j \geq p-1$. Since $\frac{1}{j}<\frac{1}{p-1}$ and using the fact that $v_p((j+1)!) < \frac{j+1}{p-1}$, we see that
	$$\frac{v_p((j+1)!)}{j} < \frac{1}{p-1} + \frac{1}{j(p-1)} \leq \frac{1}{p-1}+\frac{1}{(p-1)^2} = \frac{p}{(p-1)^2}.$$
	
	As $p \geq k+2$, we see that $\frac{p}{(p-1)^2} < \frac{1}{k}.$ By combining Cases (i) and (ii), we deduce that
		$$\max\limits_{1\leq j\leq n}\bigg\{\frac{v_p^x(b_0a_0(x))-v_p^x(b_ja_j(x))}{j}\bigg\} \leq \max\limits_{1\leq j\leq n}\bigg\{\frac{v_p((j+1)!)}{j}\bigg\} < \frac{p}{(p-1)^2} < \frac{1}{k}.$$
		Thus the slope of the rightmost edge is less than $\frac{1}{k}$. Since the slopes of the edges of the $\phi$-Newton polygon increase from left to right, the slope of each edge of the $\phi$-Newton polygon of $F(x)$ with respect to $p$ is less than $\frac{1}{k}.$ This proves our claim.
		
		Now suppose to the contrary that $F(x)$ has a factor with degree lying in the interval $[k\deg\phi(x), (k+1)\deg\phi(x))$. Then there exist polynomials $g(x), h(x) \in \Z[x]$ with $F(x) = g(x)h(x)$ and 
		\begin{equation}\label{eq:1.1}
			k\deg\phi(x) \leq \deg h(x) < (k+1)\deg \phi(x).
		\end{equation}
		Let $\sum\limits_{i=0}^{t}h_i(x)\phi(x)^i$ be the $\phi$-expansion of $h(x)$ with $h_t(x)\neq 0$. Note that $t$ is the sum of the lengths of horizontal projection of all the edges of the $\phi$-Newton polygon of $h(x)$ with respect to $p$. We have
		\begin{equation}\label{eq:1.2}
			t\deg\phi(x) \leq \deg h(x) = \deg h_t(x) + t\deg \phi(x) \leq (t+1)\deg\phi(x)-1.
		\end{equation}
		This implies that 
		$$t+1 \geq \frac{\deg h(x) +1}{\deg\phi(x)} \geq \frac{k\deg\phi(x) +1}{\deg\phi(x)} = k+\frac{1}{\deg\phi(x)},$$
		which gives 
		\begin{equation}
			t\geq k-1+\frac{1}{\deg\phi(x)} > k-1.
		\end{equation}
		We now consider an edge of the $\phi$-Newton polygon of $F(x)$ with respect to $p$ which has positive slope. Let $(a,b)$ and $(c,d)$ be two points with integer entries on such an edge. Then slope of the line joining these points is the slope of that edge, so that
		$$\frac{1}{|c-a|} \leq \frac{|d-b|}{|c-a|} < \frac{1}{k}.$$
		Thus $|c-a| > k$. Therefore any two such points on an edge with positive slope of the $\phi$-Newton polygon of $F(x)$ have their $x$-coordinates separated by a distance strictly greater than $k$. In view of Equations (\ref{eq:1.1}) and (\ref{eq:1.2}), we have
		$$t\leq \frac{\deg h(x)}{\deg \phi(x)} < k+1.$$
		The above inequality implies that $t$, which is the sum of the lengths of horizontal projections of all the edges of the $\phi$-Newton polygon of $h(x)$, is strictly less than $k+1$. In view of Lemma \ref{lm1}, the translates of the edges of the $\phi$-Newton polygon of $h(x)$ cannot be found within those edges of the $\phi$-Newton polygon of $F(x)$ which have positive slope. So the translates of the edges of the $\phi$-Newton polygon of $h(x)$ (with respect to $p$) must be among the edges of the $\phi$-Newton polygon of $F(x)$ (with respect to $p$) having $0$ slope. On the other hand, the endpoints of the edges of the $\phi$-Newton polygon of $F(x)$ having $0$ slope must be among the points $(j, v_p^x(b_{n-j}a_{n-j}(x))$ for $j\in \{0,1, \cdots, k-1\}.$ Since $k-1  < t$, these edges by themselves cannot consist of a complete collection of translated edges of the $\phi$-Newton polygon of $h(x)$ in view of Lemma \ref{lm1}. So, we have a contradiction. Thus $F(x)$ can not have a factor in $\Z[x]$ with degree lying in the interval $[k\deg\phi(x), (k+1)\deg\phi(x)).$		
	\end{proof}

	\section{Proof of Theorem \ref{1.1}.}

	We first prove the following lemma, which will be used in the proof of Theorem \ref{1.1}.

\begin{lemma}\label{1.3}
	Let $n\geq 1$ be an integer. Let $\phi(x)\in \Z[x]$ be a monic polynomial which is irreducible modulo a fixed prime $p$ dividing $n+1$. Suppose that $a_0(x), a_1(x), \cdots, a_{n-1}(x)$ belonging to $\Z[x]$ are polynomials each having degree less than $\deg \phi(x)$. Let $a_n$ be an integer with $p\nmid a_n$. Then the polynomial $f(x) = a_n \frac{\phi(x)^n}{(n+1)!} + \sum\limits_{j=0}^{n-1}a_j(x)\frac{\phi(x)^{j}}{(j+1)!}$ can not have any non-constant factor having degree less than $\deg \phi(x)$.
\end{lemma}
\begin{proof}
	It is enough to show that $F(x) = (n+1)!f(x) = a_n\phi(x)^n + \sum\limits_{j=0}^{n-1}\frac{(n+1)!}{(j+1)!} a_j(x)\phi(x)^j$ does not have a non-constant factor over $\Z$ with degree less than $\deg\phi(x)$. Since the leading coefficient of $F(x)$ is $a_n$, it is not divisible by $p$. Let $c$ denote the content of $F(x)$. As $p\nmid a_n$, we have $p\nmid c$. Now suppose to the contrary that there exists a primitive non-constant polynomial $h(x) \in \Z[x]$ dividing $F(x)$ having degree less than $\deg \phi(x)$. Then in view of Gauss Lemma, there exists $g(x)\in \Z[x]$ such that $\frac{F(x)}{c} = h(x)g(x)$. The leading coefficient of $F(x)$ and hence those of $h(x)$ and $g(x)$ are coprime with $p$. Note that $p$ divides $\frac{(n+1)!}{(j+1)!}$ for $0\leq j\leq n-1$. Therefore on passing to $\Z/p\Z$, we see that the degree of $\bar{h}(x)$ is same as that of $h(x)$. Hence $\deg\bar{h}(x)$ is positive and less than $\deg\phi(x)$. This is impossible because $\bar{h}(x)$ is a divisor of $\frac{\overline{F}(x)}{\bar{c}} = \frac{\bar{a}_n}{\bar{c}}\bar{\phi}(x)^n$ and $\bar{\phi}(x)$ is irreducible over $\Z/p\Z$. This completes the proof of the lemma.
	\end{proof}
	The following result is due to Hanson\cite{Fau}. It will be used in the sequel. We omit its proof.
\begin{lemma}\label{key}
	Let $n$ be an integer $\geq 4$, and let $k$ be an integer in $[2,\frac{n}{2}]$. Then there always exists a prime $p \geq k+2$ such that 
	$p|(n+1)n(n-1)\cdots (n-k+2)$ with the exceptions $(n,k) \in \{(8,2)\}.$
\end{lemma}
\begin{proof}[Proof of Theorem \ref{1.1}.] Let $f(x)$ be as in Theorem \ref{1.1}. In view of Lemma \ref{1.3}, it is enough to prove that $f(x)$ can not have a factor in $\Z[x]$ with degree lying in the interval $[\deg\phi(x), (\frac{n}{2}+1)\deg\phi(x))$. 

We first prove that $f(x)$ can not have a  factor in $\Z[x]$ with degree lying in the interval $[\deg\phi(x), 2\deg\phi(x))$. By hypothesis $n+1 \neq 2^u$ for any integer $u\geq 2$. Hence there always exists an odd prime $p$ which divides $n+1$. Hence by applying Theorem \ref{1.2} for $k=1$, we see that $f(x)$ can not have a  factor in $\Z[x]$ with degree lying in the interval $[\deg\phi(x), 2\deg\phi(x))$. 

 Now assume that $k\in [2,\frac{n}{2}]$ and $n\neq 8$.  Then using Lemma \ref{key} and Theorem \ref{1.2}, it follows that $f(x)$ can not have a factor in $\Z[x]$ with degree lying in the interval $[2\deg\phi(x), (\frac{n}{2}+1)\deg\phi(x))$. 
     This completes the proof of the theorem.
    \end{proof}
    \begin{remark}
    	From the proof of Theorem \ref{1.1}, it may be noted that in the case $n+1 = 2^u$ for some integer $u\geq 2$, if $f(x)$ does not have any factor in $\Z[x]$ with degree lying in the interval $[\deg\phi(x), 2\deg\phi(x))$, then $f(x)$ is irreducible over $\Q$. Further, in the case $n=8$, if $f(x)$ does not have any factor in $\Z[x]$ with degree lying in the interval $[2\deg\phi(x), 3\deg\phi(x))$, then $f(x)$ is irreducible over $\Q$.
    \end{remark}

 \end{document}